\documentclass[12pt]{amsart}

\usepackage[pdftex,colorlinks,citecolor=blue,urlcolor=blue]{hyperref}

\usepackage{amssymb}
\usepackage{graphicx,color}
\usepackage{amsmath}
\usepackage{comment}

\usepackage{tikz-cd}
\usepackage{tikz}
\usetikzlibrary{matrix}
\usepackage{cleveref}

\newtheorem{theorem}{Theorem}[section]
\newtheorem{proposition}[theorem]{Proposition}
\newtheorem{lemma}[theorem]{Lemma}

\theoremstyle{definition}
\newtheorem{definition}[theorem]{Definition}
\newtheorem{notation}[theorem]{Notation}

\newtheorem{question}[theorem]{Question}

\newtheorem{remark}[theorem]{Remark}

\newcommand{\RR}{ \ensuremath{\mathbb{R}}}

\newcommand{\po}[1]{\mathcal{#1}}

\newcommand{\conv}{\ensuremath{\mathrm{conv}}\hspace{1pt}}

\usepackage[colorinlistoftodos]{todonotes}

\makeatletter
\def\moverlay{\mathpalette\mov@rlay}
\def\mov@rlay#1#2{\leavevmode\vtop{%
		\baselineskip\z@skip \lineskiplimit-\maxdimen
		\ialign{\hfil$\m@th#1##$\hfil\cr#2\crcr}}}
\newcommand{\charfusion}[3][\mathord]{
	#1{\ifx#1\mathop\vphantom{#2}\fi
		\mathpalette\mov@rlay{#2\cr#3}
	}
	\ifx#1\mathop\expandafter\displaylimits\fi}
\makeatother

\newcommand{\lk}{{\mathrm{lk}}}

\numberwithin{equation}{section}
\title{Octahedralizing $3$-colorable $3$-polytopes}
\author[G. Codenotti]{Giulia Codenotti}
\email{giulia.codenotti@fu-berlin.de}
\address{
Institut f\"{u}r Mathematik, Freie Universit,
Arnimallee,
14195 Berlin, GERMANY
}
\author[L. Venturello]{Lorenzo Venturello}
\email{lorenzo.venturello@uni-osnabrueck.de}
\address{
Universit\"{a}t Osnabr\"{u}ck,
Fakult\"{a}t f\"{u}r Mathematik,
Albrechtstra\ss e 28a,
49076 Osnabr\"{u}ck, GERMANY
}
\keywords{Cross-polytopal complex, cross-polytope, polytopal subdivisions,  3-colorable 3-polytopes}
\date{\today}

\thanks{
The first author was supported by the Einstein Foundation Berlin. The second author was supported by the German Research Council DFG GRK-1916.
}

\subjclass[2010]{52B10, 52B12}

\begin{document}
\maketitle
\begin{abstract}
    We investigate the question of whether any $d$-colorable simplicial $d$-polytope can be octahedralized, i.e., can be subdivided to a $d$-dimensional geometric cross-polytopal complex. We give a positive answer in dimension $3$, with the additional property that the octahedralization introduces no new vertices on the boundary of the polytope.
\end{abstract}

\section{Introduction and preliminaries}
The study of triangulations is a central theme in discrete geometry and beyond. In words, a triangulation of a polytope is a decomposition as a union of simplices which intersect properly along common faces, and it is not hard to see that any polytope can be triangulated (see \cite{DRS} for more about triangulations). In this paper we consider a different decomposition for simplicial $d$-polytopes which are \emph{balanced}, i.e., their graph is $d$-colorable, in the classical graph theoretic sense. Clearly since the graph of any $(d-1)$-simplex is the complete graph on $d$ vertices, $d$ is the minimum chromatic number that the graph of any simplicial $d$-polytope can have. Balanced simplicial complexes were introduced by Stanley \cite{St79} and recently they have gained attention from the point of view of face enumeration \cite{KN, JKM, Juhnke:Murai:Novik:Sawaske, Ven}. For results of a more topological flavour regarding balancedness and colorings we refer to \cite{Fisk1977, Izmestiev:Joswig, IKN, JKV}. Under many perspectives it appears that, when dealing with balanced complexes, the cross-polytope, which is easily seen to be balanced, plays the fundamental role played by the simplex in the setting of arbitrary complexes. The starting point of this paper is a lemma in \cite{IKN}, where the authors describe a procedure to systematically convert a balanced simplicial complex into a \emph{(combinatorial) cross-polytopal complex}; that is to say, a pure regular CW-complex in which the boundary of each maximal cell is combinatorially isomorphic to the boundary complex of a cross-polytope. We investigate a geometric version of this statement, which asks for the existence of a geometric cross-polytopal complex decomposing (in $d=3$ "octahedralizing") balanced $d$-polytopes. We proceed now with some basic definitions (see \cite{ZieglerBook} for basics on polytopes).\\

The \emph{regular $d$-dimensional cross-polytope} $\Diamond_d$ is the polytope $\conv(\pm e_1,\dots, \pm e_d) \subseteq \RR^d$, where $\{e_i\}^d_{i=1}$ is the canonical basis of $\RR^d$. 
We say that two $d$-polytopes are \emph{combinatorially equivalent} (and we denote it by $\simeq$) if their face lattices are isomorphic. For the rest of this paper we will call \emph{$d$-dimensional cross-polytope} any convex polytope combinatorially isomorphic to the regular cross-polytope. A polytope $\mathcal{P}$ is \emph{$k$-colorable} if its graph is $k$-colorable in the classical graph-theoretic sense, i.e., if there exists a map $\kappa:V(\mathcal{P})\rightarrow[k]$ such that $\kappa(V)\neq\kappa(W)$ whenever $V$ and $W$ are the vertices of an edge. Note that if $\mathcal{P}$ is a simplicial $d$-polytope then the graph of a facet is the complete graph on $d$ vertices, and so a simplicial $d$-polytope cannot be $k$-colorable for any $k<d$. In the literature, a $d$-colorable simplicial $d$-polytope, or more generally a $d$-colorable $(d-1)$-dimensional simplicial complex, is often called \emph{balanced}.

A \emph{(geometric) polytopal complex} $\Delta$ in $\RR^d$  is a collection of polytopes in $\RR^d$ such that 
\begin{itemize}
    \item if $\mathcal{F}_i, \mathcal{F}_j \in \Delta$, then $\mathcal{F}_i \cap \mathcal{F}_j$ is a (possibly empty) face of both $\mathcal{F}_i$ and $\mathcal{F}_j$;
    \item if $\mathcal{G}$ is a face of $\mathcal{F}$, and $\mathcal{F} \in \Delta$, then $\mathcal{G} \in \Delta$.
\end{itemize}
Elements of a polytopal complex $\Delta$ are called \emph{cells}. We denote by $f_i(\Delta)$ the number of $i$-dimensional cells of the complex and a complex is called \emph{pure} if all maximal cells have the same dimension. If all cells are simplices, the complex is called a simplicial complex. We are interested in a different specialization of polytopal complexes:
\begin{definition}
A \emph{(geometric) cross-polytopal complex} $\Delta$ in $\RR^d$ is a pure polytopal complex where all maximal cells are cross-polytopes.
We call the \emph{support} of the complex $\Delta$ the set $\left|\Delta\right|=\bigcup_{\mathcal{C}\in\Delta}\mathcal C$ and we say that $\Delta$ is a  \emph{cross-polytopal subdivision} of  $\left|\Delta\right|$. 
\end{definition}
Moreover, we denote with $\partial\Delta$ the \emph{boundary} of $\Delta$, that is the simplicial complex generated by $(d-1)$-dimensional cells that belong to a unique maximal cell of $\Delta$. 
A cross-polytopal subdivision $\Delta$ of a simplicial polytope $\mathcal{P}$ such that $\partial\Delta\simeq\partial\mathcal{P}$ is called \emph{proper}. Somehow informally, we can use the word \emph{octahedralization} to refer to cross-polytopal subdivisions of $3$-polytopes.

With these definitions, we can precisely formulate the question at the heart of this paper:
\begin{question}\label{quest:main}
 Does every balanced $d$-polytope have a proper cross-polytopal subdivision?
\end{question}

The balanced property in \Cref{quest:main} is necessary, as the following proposition shows.

\begin{proposition}\label{prop:bal_last}
    Let $\Delta$ be a $d$-dimensional (combinatorial) cross-polytopal complex whose support is PL-homeomorphic to a $d$-ball. Then the $(d-1)$-skeleton of $\Delta$, i.e., the simplicial complex $\{F\in \Delta: \dim(F)<d\}$, is a balanced $(d-1)$-dimensional simplicial complex.
\end{proposition}
\begin{proof}
    Let $\Delta'$ be the combinatorial $d$-ball (i.e., a simplicial complex PL-homeomorphic to the $d$-simplex) obtained by (stellar) subdividing all $d$-cells of $\Delta$ introducing a point $C_{F}$, for each $d$-cell $F$. The interior $(d-2)$-faces of $\Delta'$ can be subdivided in two sets:
    \begin{itemize}
        \item The faces containing one of the vertices $C_F$;
        \item The $(d-2)$-faces of $\Delta$.
    \end{itemize}
    For each $(d-2)$-faces $G$ we consider the \emph{link} of $G$ in $\Delta'$, i.e.,
    \[\lk_{\Delta'}(G):=\{H\in\Delta': G\cap H=\emptyset, G\cup H \in\Delta'\}.\]
    The link of a $(d-2)$-face containing one of the vertices $C_F$ is a $1$-sphere contained in the corresponding cross-polytopal cell, and therefore it is a $4$-cycle. Let $G$ be any $(d-2)$-face not containing any $C_i$ and assume that $G$ lies in the cross-polytopal cells $\{F_i:i=1,\dots k\}$. Then the link of $G$ is the $2k$-cycle given by the vertices $C_{F_i}$ connected with the only pair of antipodal vertices of $F_i$ not intersecting $F_i$. By \cite[Corollary 11]{Jos} we know that a triangulation of a ball in which all links of interior $(d-2)$-faces are even polygons is balanced. In particular $\Delta'$ is a $(d+1)$-colorable simplicial $d$-ball and the (unique up to permutation of the images) coloring assigns to the vertices $C_F$ the same color, say $d+1$. Since it is elementary to check that the $C_F$'s are the only vertices colored with $d+1$, the claim follows. 
\end{proof}
In particular, since full dimensional subcomplexes of balanced simplicial complexes are balanced, we have that $\partial\Delta$ is a balanced $(d-1)$-sphere.
\begin{remark}
    To be precise, \cite[Corollary 11]{Jos} shows that a simply connected combinatorial $d$-manifold without boundary in which all links of $(d-2)$-faces are even polygons is balanced. The arguments in the proof of \cite[Theorem 8]{Jos} can be extended to manifolds with boundary. Alternatively, one can consider the combinatorial $d$-sphere obtained coning over the boundary of $\Delta'$ in the proof of \Cref{prop:bal_last}, and then argue with \cite[Corollary 11]{Jos}.
\end{remark}
\begin{figure}[h]
\includegraphics[scale=1]{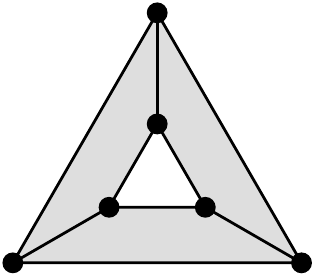}
\caption{A $2$-dimensional cross-polytopal complex that is $3$-colorable but not $2$-colorable.}\label{fig:strongly_conn_not_balanced}
\end{figure}
Observe that if $\left|\Delta\right|$ is not simply connected, \Cref{prop:bal_last} does not hold, as it can be seen in the example in \Cref{fig:strongly_conn_not_balanced}. 

The structure of the paper is as follows. After outlining in \Cref{sect:bipyramid} a general strategy to study \Cref{quest:main} in arbitrary dimension, \Cref{sect:schlegel,sect:dim3} are devoted to answering it, positively, in dimension $d=3$.

Let us recall some notation in dimension $3$: the regular $3$-dimensional cross-polytope is called \emph{regular octahedron}, and we thus denote by \emph{octahedron} any $3$-polytope which is combinatorially equivalent to the $3$-dimensional cross-polytope. A \emph{tetrahedron} is a $3$-dimensional simplex. A summary of our results in dimension $3$ is:
\begin{theorem}\label{thm:3-colorable}
    Let $\mathcal{P}$ be a balanced $3$-polytope. Then there exists a proper cross-polytopal subdivision of $\mathcal{P}$. In particular, there is one such subdivision $\Delta$ with $f_3(\Delta)=23(f_0(\mathcal{P})-2)$.
\end{theorem}

If we admit subdivisions which are not proper, then we can prove the following.
\begin{theorem}\label{thm:3simplex}
    Any tetrahedron has a (non-proper) cross-polytopal subdivision $\Delta$ such that $f_3(\Delta)=23$ and $\partial\Delta\simeq\partial\Diamond_3$.
\end{theorem}

\section{A strategy for octahedralizations via bipyramids}\label{sect:bipyramid}
The following lemma, which employs an idea discussed in \cite[Lemma 3.6]{IKN}, shows a first attempt to reduce \Cref{quest:main} to the problem of decomposing certain generalized bipyramids into cross-polytopes. These bipyramids arise from a natural matching on the $d$-simplices of a triangulation of a balanced simplicial $d$-polytope induced by the coloring. 

\begin{lemma}\label{lem:triangulation}
Let $\mathcal{P}$ be a balanced $d$-polytope. Then $\mathcal{P}$ can be triangulated in a way such that:
\begin{itemize}
    \item The $d$-simplices in the triangulation can be partitioned in pairs sharing a facet; this facet is a $(d-1)$-simplex, which we call the \emph{equatorial simplex};
    \item For each equatorial simplex $\po{E}$, there exists a flag of faces $F_0\subset F_1 \subset \dots \subset F_{d-2}$ of $\po{E}$, with $\dim(F_i)=i$, such that: if any $F_i$ is contained in another pair of $d$-simplices then it is contained in their equatorial simplex. 
\end{itemize}
\end{lemma}
\begin{proof}
    Let $\Phi$ be any triangulation of $\po{P}$ whose graph is $(d+1)$-colorable, and whose boundary coincides with the boundary of $\po{P}$. The simplest example of such a triangulation is the one whose cells are the cones over every facet of $\po{P}$ from a fixed interior point. We color the vertices of $\Phi$. Since $\po{P}$ is $d$-colorable, we can assume that all the vertices of color $d+1$ are in the interior. Hence no $(d-1)$-simplices whose vertices are colored with colors $\{2,\dots,d+1\}$ are on the boundary of $\po{P}$. Each such $(d-1)$-simplex is therefore a facet of exactly two $d$-simplices of $\Phi$, which we take as our pairs. The $(d-1)$-simplices colored with $\{2,\dots,d+1\}$ are thus the equatorial simplices, and indeed each simplex of $\Phi$ contains exactly one such simplex. Finally, each equatorial simplex has a unique $i$-face $F_i$ colored using colors in $\{d+1-i,\dots,d+1\}$, for $i=0,\dots,d-2$. The flag $F_0\subset F_1 \subset \dots \subset F_{d-2}$ so defined satisfies the second condition.
\end{proof}
\begin{remark}
    Observe that \Cref{lem:triangulation} shows that every balanced $d$-dimensional simplicial polytope has an even number of facets for every $d$, even when $d$ is even. More generally, this is true for every balanced $d$-dimensional \emph{pseudomanifold}, i.e., a simplicial complex with the property that every $(d-1)$-dimensional face is contained exactly in $2$ facets. This fact can be deduced by other means.
\end{remark}
    It is important to note that \Cref{lem:triangulation} guarantees that we can pair up the $d$-simplices of our triangulation, but the union of the two simplices is not in general convex. In other words the \emph{dual graph} of a balanced polytope, which is a bipartite graph (see e.g., \cite[Proposition 6]{Jos}), admits a perfect matching.
    The lemma thus shows that finding a balanced subdivision of the following class of objects is enough to positively answer \Cref{quest:main}: 
\begin{definition}
\label{def:bipyr}
    A \emph{generalized bipyramid} $\mathcal{B}$ is the union of two $d$-simplices $\mathcal{S}_1$ and $\mathcal{S}_2$ which intersect along a $(d-1)$-simplex $\po{E}$, called the \emph{equatorial simplex}, which is a face of both. Observe that $\mathcal{B}$ need not be a convex polytope. We fix a distinguished flag of faces of $\po{E}$, $F_0 \subsetneq F_1 \subsetneq \dots \subsetneq F_{d-2}\subsetneq F_{d-1}=\po{E}$. This data is part of the definition of a generalized bipyramid. 
\end{definition}
\begin{figure}[h]
\includegraphics[scale=0.5]{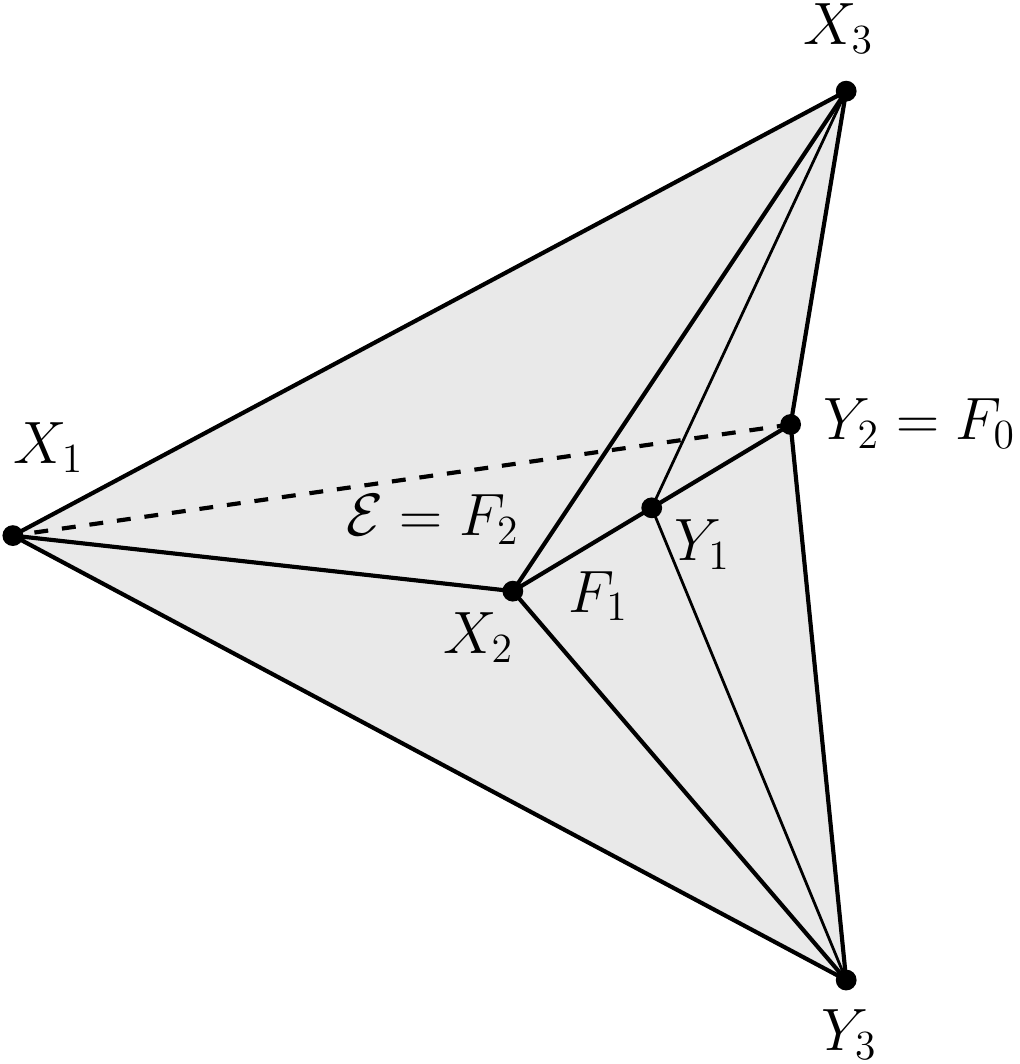}
\caption{A generalized bipyramid whose boundary is isomorphic to $\partial\Diamond_3$.}
\label{fig:generalized}
\end{figure}
    
We want to think of the generalized bipyramid as a degenerate cross-polytope, obtained by deforming an embedded codimension $1$ cross-polytope into a simplex. \Cref{fig:generalized} depicts the case $d=3$, where the $4$-cycle is transformed into a $3$-cycle and the two remaining vertices are perturbed. To make this precise, we consider the following triangulation of $\partial\po{B}$:
\[\Gamma=\text{sd}_{F_0}\circ\text{sd}_{F_1}\circ\dots\circ\text{sd}_{F_{d-2}}(\partial\po{B}) \]
where $\text{sd}_{F}(\Delta)$ is the \emph{stellar subdivision} of as simplicial complex $\Delta$ at a face $F$, that is, the simplicial complex obtained from $\Delta$ replacing all the faces containing $F$ (called the \emph{star} at $F$) with those given by the union of the barycenter of $F$ with every face of the boundary of its star. 
In words, to obtain $\Gamma$ we iteratively stellar subdivide $\partial{\po{B}}$ at the faces $F_i$ in decreasing dimension. See \Cref{fig:generalized} for a three dimensional example.

\begin{notation}
Here and for the rest of this article, we denote the vertices of $\Gamma$ as follows: for $1\leq i \leq d$ we let $X_{i}$ be the vertex of the $F_{d-i}$ not in $F_{d-i-1}$ and $Y_{i}$ be the barycenter of $F_{d-i-1}$, and we denote with $X_d$ and $Y_d$ the vertices not in the equatorial simplex $\mathcal{E}$. In what follows, we will assume the boundary of all generalized bipyramid to be subdivided as above, and with a slight abuse of notation will refer to the vertices of $\Gamma$ as vertices of the bipyramid $\po{B}$.
\end{notation}

\begin{remark}\label{rem:degeneration}
It is easy to see that the triangulation $\Gamma$ thus obtained is a simplicial complex isomorphic to the boundary of the $d$-cross-polytope, with the pairs $X_i,Y_i$ as opposite vertices. This is what allows us to think of $\po{B}$ as a degenerate cross-polytope.
\end{remark}
The affine dependence between the points $X_i,Y_i$ is described in the following lemma.

\begin{lemma}\label{rem:barycenters}
    The points $X_1,Y_1,\dots,X_{i},Y_{i}$ lie in a $i$-dimensional linear space $\mathcal{L}_{i}$, for $i=1,\dots,d$, with $\mathcal{L}_1 \subsetneq \mathcal{L}_2 \subsetneq \dots \subsetneq \mathcal{L}_{d-1}$; further, $X_{i}$ and $Y_{i}$ lie in opposite halfspaces defined by the hyperplane $\mathcal{L}_{i-1}$ in $\mathcal{L}_{i}$. 
\end{lemma}
\begin{proof}
    Since $O$ is the barycenter of the equatorial simplex the points $O$, $X_1$ and $Y_1$ are aligned. We denote this linear space with $\mathcal{L}_1$. For every $i=2,\dots, d-1$ the point $Y_{i}$ is the barycenter of the simplex with vertices $\{X_{i+1},\dots, X_{d-1}\}\cup\{Y_{d-1}\}$, which implies that $Y_i$ lies on the line through $X_i$ and $Y_{i-1}$ and therefore $\mathcal{L}_{i}:=\text{span}(X_1,Y_1, \dots, X_{i},Y_{i})$ has $\dim(\mathcal{L}_i)=\dim(\mathcal{L}_{i-1})+1 $, for every $i=2,\dots,d-1$.
    Moreover, for $2\leq i \leq d-2$, $Y_{i-1}$ lies on the segment $X_{i},Y_{i}$:
    indeed, $Y_{i-1}=\sum_{j=2}^{i} \frac{X_{j}}{i}+\frac{Y_{d-1}}{i}=\frac{X_i}{i}+\frac{i-1}{i}Y_{i}$ holds by definition of barycenter. Since $X_{i}$ and $Y_{i}$ do not lie on $\mathcal{L}_{i-1}$, they must lie on either side.  
\end{proof}

\section{The Schlegel diagram of the 24-cell}\label{sect:schlegel}


The starting point of our decomposition is a regular convex $4$-polytope called the \emph{24-cell}, which can be realized as the convex hull of all vectors in $\RR^4$ with exactly two zero entries and two entries in $\{1,-1\}$.  Its boundary consists of 24 octahedral cells and therefore its Schlegel diagram provides a subdivision of the regular octahedron into 23 octahedra. This subdivision can be described as follows.
Consider the regular \emph{cuboctahedron} $\mathcal{H}=\conv(\{\lambda e_i+\mu e_j : \lambda,\mu\in \{-1,1\}, i\neq j\in\{1,2,3\}\})$.
Placing a scaled copy of $\mathcal{H}$ inside a regular octahedron $\po{C}_o$ (the subscript stands for "outer"), we observe that each square face of the cuboctahedron lies on a plane orthogonal to a line through antipodal points in the octahedron. Next, we add a second scaled regular octahedron $\po{C}_i$ (the "inner" octahedron) inside the cuboctahedron, again with the same center, and whose faces lie pairwise on planes parallel to the faces of the outer octahedron.\\
In this configuration to each of the 6 square faces of $\mathcal{H}$ correspond a vertex of $\po{C}_o$ and a vertex of $\po{C}_i$, and to each of the 8 triangular faces of $\mathcal{H}$ corresponds a $2$-face of $\po{C}_o$ and a $2$-face of $\po{C}_i$. This correspondence naturally gives rise to a cross-polytopal decomposition of the outer octahedron $\po{C}_o$ in 23 octahedra, which we divide in 4 types, illustrated in \Cref{fig_1}:
\begin{figure}[h]
\includegraphics[scale=0.9, trim={7cm 11cm 7cm 11cm}, clip]{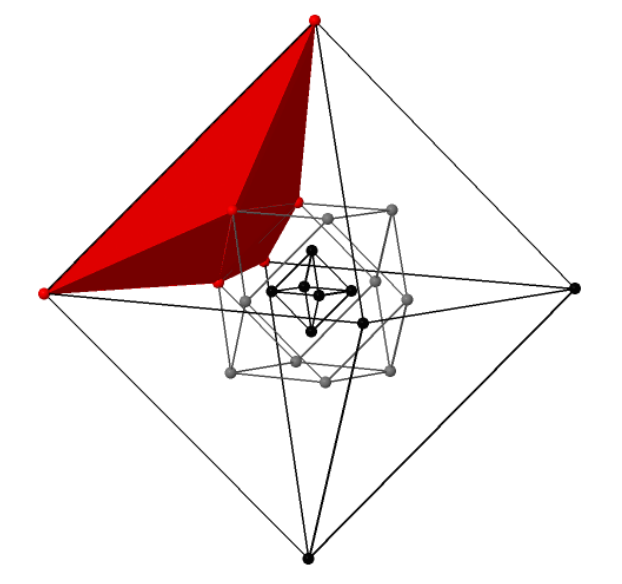}\qquad
\includegraphics[scale=0.9, trim={7cm 11cm 7cm 11cm}, clip]{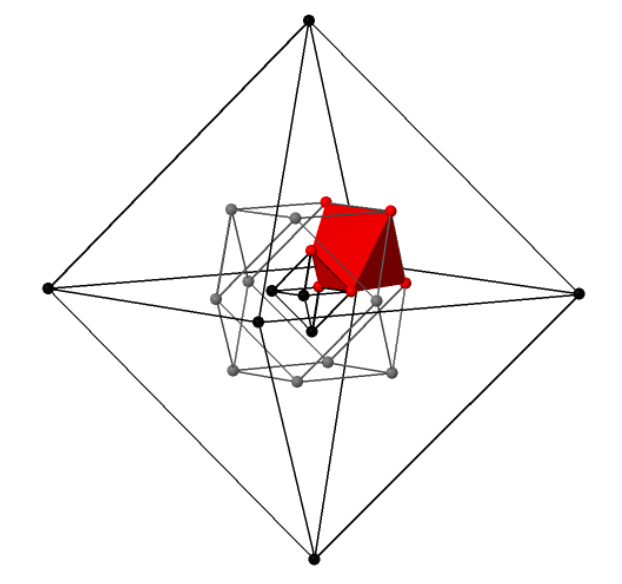}\\
\includegraphics[scale=0.9, trim={7cm 11cm 7cm 11cm}, clip]{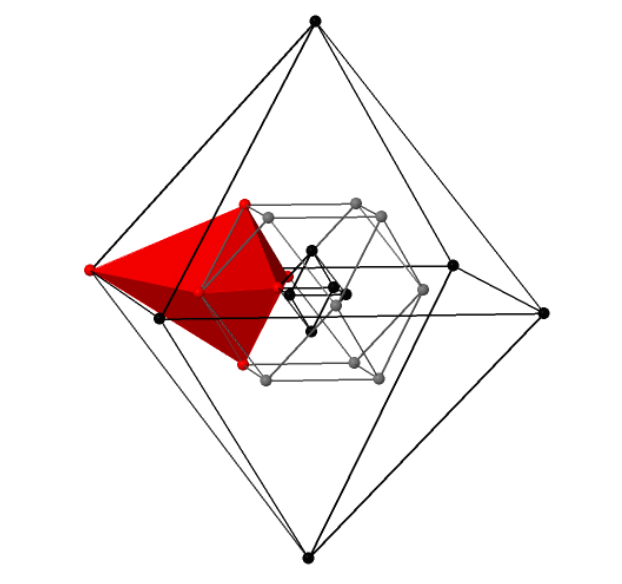}\qquad
\includegraphics[scale=0.9, trim={7cm 11cm 7cm 11cm}, clip]{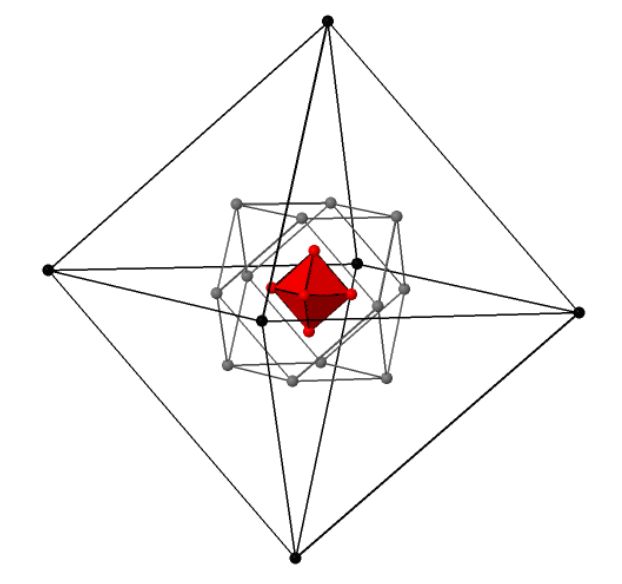}
\caption{The four types of octahedra in the Schlegel diagram of the 24-cell.}
\label{fig_1}
\end{figure}
\begin{itemize}
\item \textbf{Type 1}: Octahedra obtained as the convex hull of a triangular face of $\mathcal{H}$ and a face in $\po{C}_o$. There are 8 octahedra of type 1.
\item \textbf{Type 2}: Octahedra obtained as the convex hull of a triangular face of $\mathcal{H}$ and a face in $\po{C}_i$. There are 8 octahedra of type 2.
\item \textbf{Type 3}: Octahedra obtained as the convex hull of a square face of $\mathcal{H}$ and the two corresponding vertices of $\po{C}_o$ and $\po{C}_i$. There are 6 octahedra of type 3.
\item \textbf{Type 4}: The inner octahedron $\po{C}_i$.
\end{itemize}

\section{The subdivision}\label{sect:dim3}
This section is devoted to proving the following proposition, which allows us to prove \Cref{thm:3-colorable}.
\begin{proposition}\label{prop:bipyramid}
    Let $\mathcal{B}\subseteq\RR^3$ be a $3$-dimensional generalized bipyramid. There exists a (non-proper) geometric cross-polytopal sudivision $\Delta$ of $\mathcal{B}$ such that $\partial\Delta\simeq\partial\Diamond_3$. In particular there is one such $\Delta$ with $f_3(\Delta)=23$.
\end{proposition}

To prove this, we will construct a subdivision mimicking the one outlined for the regular octahedron in \Cref{sect:schlegel}. We first give an outline of the proof strategy. The lemmas we prove hold in any dimension, and so we state and prove them in that level of generality. Let $\po{B}$ be a generalized bipyramid with the origin $O$ as the barycenter of the equatorial simplex and let $\{X_i,Y_i\}_{i=1}^{d}$ be its vertices, or more precisely the vertices of the triangulation $\Gamma$ of its boundary (see \Cref{def:bipyr}).
\begin{itemize}
	\item In \Cref{lem:from_simplex_to_cp}, we construct a convex $d$-dimensional cross-polytope $\po{C}$ with vertices $\{V_i,W_i\}_{i=1}^{d}$, where vertex  $V_i$ (resp. $W_i$) lies on the segment joining $O$ and the vertex $X_i$ of $\po{B}$ ($Y_i$ resp.). For any $0 < \epsilon\leq 1$ the polytope $\epsilon\po{C}$ is a  $d$-dimensional cross-polytopes contained in $\po{B}$. In the proof of \Cref{thm:3-colorable}, $\epsilon\po{C}$ will play the role of the octahedron of type 4.
	
	\item Next, we show that we can choose a point $P_e$ in the interior of each edge $e$ of $\po{C}$ so that, for any vertex $V_i$ or $W_i$ of $\po{C}$, the polytopes $\mathcal{Q}_{V_i}=\conv(V_i,O,\{P_e: V_i\in e\})$ and $\mathcal{Q}_{W_i}=\conv(W_i,O,\{P_e: W_i\in e\})$ are $d$-dimensional cross-polytopes. This is proved in \Cref{lem:cubeoctahedron}.
	
	\item In \Cref{lem:type3}, we consider a modification of $\po{Q}_{V_i}$ and $\po{Q}_{W_i}$, namely $\po{Q}'_{V_i}=\conv(X_i, \epsilon V_i,\{P_e: V_i\in e\})$ and $\po{Q}'_{W_i}=\conv(Y_i, \epsilon W_i,\{P_e: W_i\in e\})$, and show that these are also $d$-dimensional cross-polytopes. In the $3$-dimensional setting, these octahedra correspond to type 3 octahedra.
	
	\item Finally, in \Cref{lem:type12} we construct the octahedra of type 1 and 2: we show that, if $\po{F}$ is any $(d-1)$-face of $\po{C}$, and $\po{F}$ and $\po{F'}$ the corresponding $(d-1)$-faces of $\po{B}$ and $\epsilon\po{C}$, for any choice of $d$ points $P_{\po{G}}$, each on a $(d-2)$-face $\mathcal{G}$ of $\po{F}$, the polytopes $\conv(\po{F},\{P_{\po{G}}\})$ and $\conv(\po{F'},\{P_{\po{G}}\})$ are $d$-dimensional cross-polytope.
\end{itemize}
To conclude the proof of \Cref{thm:3-colorable}, we observe that, in dimension $d=3$, the octahedra described above fit together to decompose the bipyramid. 

We begin with a series of lemmas. The first, though immediate, will be useful throughout the proof, since it presents a general fact on the combinatorial structure of the cross-polytope. 
\begin{lemma}\label{lem:antipodal}
    Let $\po{P}$ be a simplicial $d$-polytope on $2d$ vertices partitioned in $d$ pairs, such that each pair is not an edge of $\mathcal{P}$. Then $\po{P}$ is combinatorially isomorphic to a $d$-dimensional cross-polytope.
\end{lemma}
\begin{proof}
    We want to show that any set of vertices containing exactly one vertex of every pair (\emph{good set}) is a facet. Certainly any facet of $\po{P}$ must contain exactly one vertex of every pair. Let $F$ be a facet and $a \in F$ any vertex of $F$. The ridge $G = F\setminus \{a\}$ is in exactly two facets, and there are only two good sets containing $G$, $F$ and $F\setminus \{a\} \cup \{b\}$, where $b$ is paired with $a$. Thus $F\setminus \{a\} \cup \{b\}$ must also be a facet. In this way we can iteratively obtain that any good set is a facet.  
\end{proof}

We can now begin to prove the lemmas necessary for the proof that were outlined above.
\begin{lemma}\label{lem:from_simplex_to_cp}
	Let $\po{B}$ be a generalized bipyramid with vertices $\{X_1,Y_1,\dots,X_d,Y_d\}$ and let $O$ be the barycenter of its equatorial simplex. There exist a configuration of points $V_1,W_1, \dots, V_{d},W_{d}$, with $V_i$ on the segment $OX_i$ and $W_i$ on the segment $OY_i$, such that $\conv(V_1,W_1, \dots, V_d,W_{d})$ is a $d$-dimensional cross-polytope.  	
\end{lemma}	
\begin{proof}
    We place pairs of points on the segments $OX_{i}$ and $OY_{i}$ for $i=1,\dots,d$ and show that at each step their convex hull is a cross-polytope of increasing dimension. First we choose any two points $V_1$ and $W_1$ on the segments $OX_{1}$ and $ OY_{1}$. Remember from \Cref{rem:barycenters} that $  X_1,Y_1,X_2,Y_2$ lie in a $2$-dimensional linear space $\mathcal{L}_2$, with $X_2$ and $Y_2$ on either side of $\mathcal{L}_1$. By continuity, there exists an open ball $B_2\subseteq \mathcal{L}_2$ such that, for any choice of points $V_{2}$ and $W_{2}$ on the intersection between the segments $OX_{2}$ and $OY_{2}$ and $B_2$, the segment $V_{2}W_{2}$ intersects the segment $V_1W_1$ in the interior. Hence $\conv(V_1,W_1,V_2,W_2)$ is a quadrilateral (indeed a $2$-dimensional cross-polytope) whose interior contains $O$. Iteratively, assume $\conv(V_1,W_1,\dots,V_i,W_i)$ is an $i$-dimensional cross-polytope. Since it contains $O$ in the interior and $\mathcal{L}_{i}$ separates (in $\mathcal{L}_{i+1}$) $X_{i+1}$ and $Y_{i+1}$, there exists a ball $B_{i+1}$ in $\mathcal{L}_{i+1}$ such that for any choice of points $V_{i+1}$ and $W_{i+!}$ on $OX_{i+1}\cap B_{i+1}$ and $OY_{i+1}\cap B_{i+1}$ the segment $V_{i+1}W_{i+1}$ intersects the polytope $\conv(V_1,W_1,\dots,V_i,W_i)$ in the interior, and hence that $\conv(V_1,W_1,\dots,V_{i+1},W_{i+1})$ is a $(i+1)$-dimensional cross-polytope.
\end{proof}
\begin{figure}[h]
\includegraphics[scale=0.7]{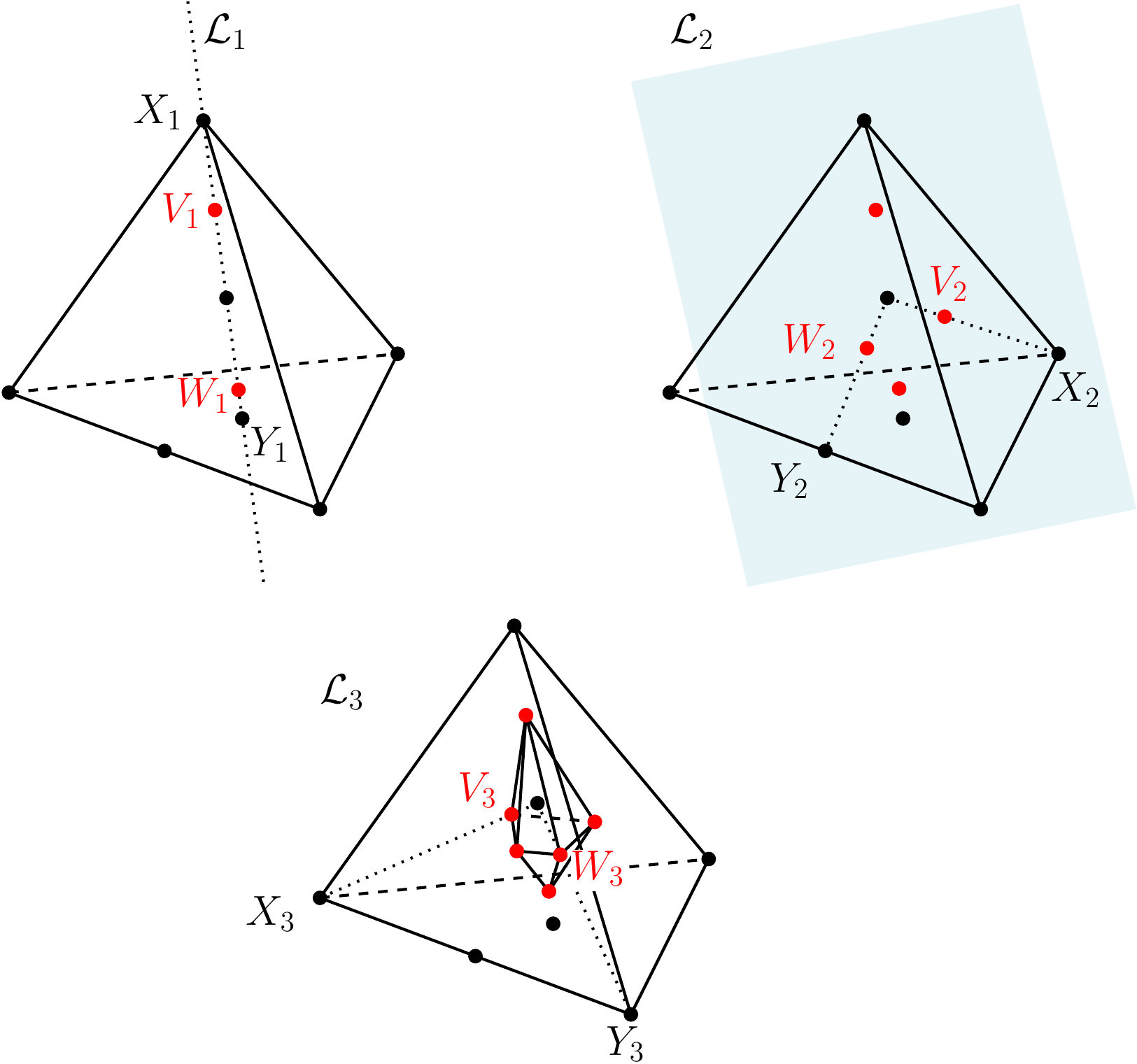}
\caption{The construction in \Cref{lem:from_simplex_to_cp} for $d=4$.}
\end{figure}

We let $\mathcal{C}$ be the cross-polytope $\mathcal{C}:=\conv(V_1,W_1,\dots,V_d,W_d)$ constructed in \Cref{lem:from_simplex_to_cp}.

\begin{lemma}\label{lem:cubeoctahedron}
    For every  edge $e$ of $\po{C}$, we can choose a point $P_e$ in the interior of $e$ such that, for any vertex $V_i$, $W_i$ of $\po{C}$, $\mathcal{Q}_{V_i}:=\conv(V_i,O,\{P_e: V_i\in e\})$ and $\mathcal{Q}_{W_i}:=\conv(W_i,O,\{P_e: W_i\in e\})$ are convex $d$-dimensional cross-polytopes.
\end{lemma}
\begin{proof}
    
    
    It follows from the proof of \Cref{lem:from_simplex_to_cp} that the points $\{V_{1}, W_1, \cdots,V_i, W_{i}\}$ lie on a $i$-dimensional linear space $\mathcal{L}_i$, and each $\mathcal{L}_i$ separates $V_{i+1}$ and $W_{i+1}$ in $\mathcal{L}_{i+1}$.
    
    The proof of this statement is by induction: we choose points $P_e$ on edges $e$ contained in $\mathcal{L}_k$ but not in $\mathcal{L}_{k-1}$, and prove by induction on $k$ that for any $i \leq k$ $\mathcal{Q}^{k}_{V_i}:=\conv(V_i,O,\{P_e: V_i\in e, e \in \mathcal{L}_k\}) \subset \mathcal{L}_k$ is a $k$-dimensional cross-polytope (analogously $\mathcal{Q}^{k}_{W_i}$). Since $\mathcal{Q}^{d}_{V_i}= \mathcal{Q}_{V_i}$, this will prove the lemma.
    
    For $k=1$, the statement above is trivially true, since $\mathcal{Q}^{1}_{V_1}$ and $\mathcal{Q}^{1}_{W_1}$ are the segments with endpoints the origin and $V_1$, $W_1$ respectively.
    
    Now suppose $k>1$. By inductive hypothesis, we have picked points $P_e$ on all edges contained in $\mathcal{L}_{k-1}$. We now want to choose points $P_e$ on the new edges in $\mathcal{L}_k$, that is, we must pick points on the edges connecting $\{V_1, W_1, \cdots, V_{k-1}, W_{k-1}\}$ and $\{V_k, W_k\}$. To do so, we choose $(d-1)$-dimensional affine spaces $\mathcal{A}_{k-1}$ and $\mathcal{A}'_{k-1}$, translations of the linear space $\mathcal{L}_{k-1}$, respectively towards $V_k$ and $W_k$, such that they intersect $\mathcal{C}$ in its interior; we let the points $P_e$ be the intersection of the edges $e$ with these affine spaces. We want to show that $\mathcal{A}_{k-1}$ and $\mathcal{A}'_{k-1}$ can be chosen so that $\mathcal{Q}^{k}_{V_i}$ is a $k$-dimesional cross-polytope for all $i \leq k$.
    We split this statement into the following two claims:\\
    \textbf{Claim 1:}
    $\mathcal{Q}^{k}_{V_k}$ and $\mathcal{Q}^{k}_{W_k}$ are $k$-dimensional cross-polytopes for any choice of $\mathcal{A}_{k-1}$ and $\mathcal{A}'_{k-1}$.\\
    \textbf{Claim 2:} 
    For $i<k$, $\mathcal{Q}^{k}_{V_i}$ and $\mathcal{Q}^{k}_{W_i}$ are $k$-dimesional cross-polytopes if $\mathcal{A}_{k-1}$ and $\mathcal{A}'_{k-1}$ are chosen sufficiently close to $\mathcal{L}_k$.\\
    \emph{Proof of Claim 1}. It suffices to observe that $\conv(\{P_e: V_i \in e, e \subset \mathcal{L}_k\})$ is a translation and dilation of $\mathcal{C}\cap \mathcal{L}_{k-1}$, and hence a $(k-1)$-dimensional cross-polytope. Since for any choice of $\mathcal{A}_{k}$ the segment $V_kO$ intersects the interior of $\conv(\{P_e: V_i \in e, e \subset \mathcal{L}_k\})$, \Cref{lem:antipodal} guarantees that $\mathcal{Q}^{k}_{V_i}$ is a $k$-dimensional cross-polytope. In the same way we can prove the claim for $\mathcal{Q}^{k}_{W_k}$.\\
    \emph{Proof of Claim 2}.
    By construction we have $\mathcal{Q}^{k-1}_{V_i}=\mathcal{Q}^{k}_{V_i} \cap \mathcal{L}_{k-1}$, which by inductive hypothesis is a $(k-1)$-dimensional cross-polytope. $\mathcal{Q}^{k}_{V_i}$ has two new vertices $P$ and $Q$, which are the points of intersection of $\mathcal{A}_k$ and $\mathcal{A}'_k$ respectively with the edges $V_iV_k$ and $V_iW_k$. By \Cref{lem:antipodal}, it is sufficient to show is that if we choose $\mathcal{A}_k$ and $\mathcal{A}'_k$ sufficiently close to $\mathcal{L}_k$, the segment $PQ$ will intersect the interior of $\mathcal{Q}^{k-1}_{V_i}$. This is true by a continuity argument, since $PQ$ intersects the interior of $\mathcal{C} \cap \mathcal{L}_{k-1}$, and when the affine spaces coincide with $\mathcal{L}_k$, we have $P=Q=V_i$. In the same way we can prove the claim for $\mathcal{Q}^{k}_{W_i}$.

\end{proof}

\begin{figure}[h]
\includegraphics[scale=0.75]{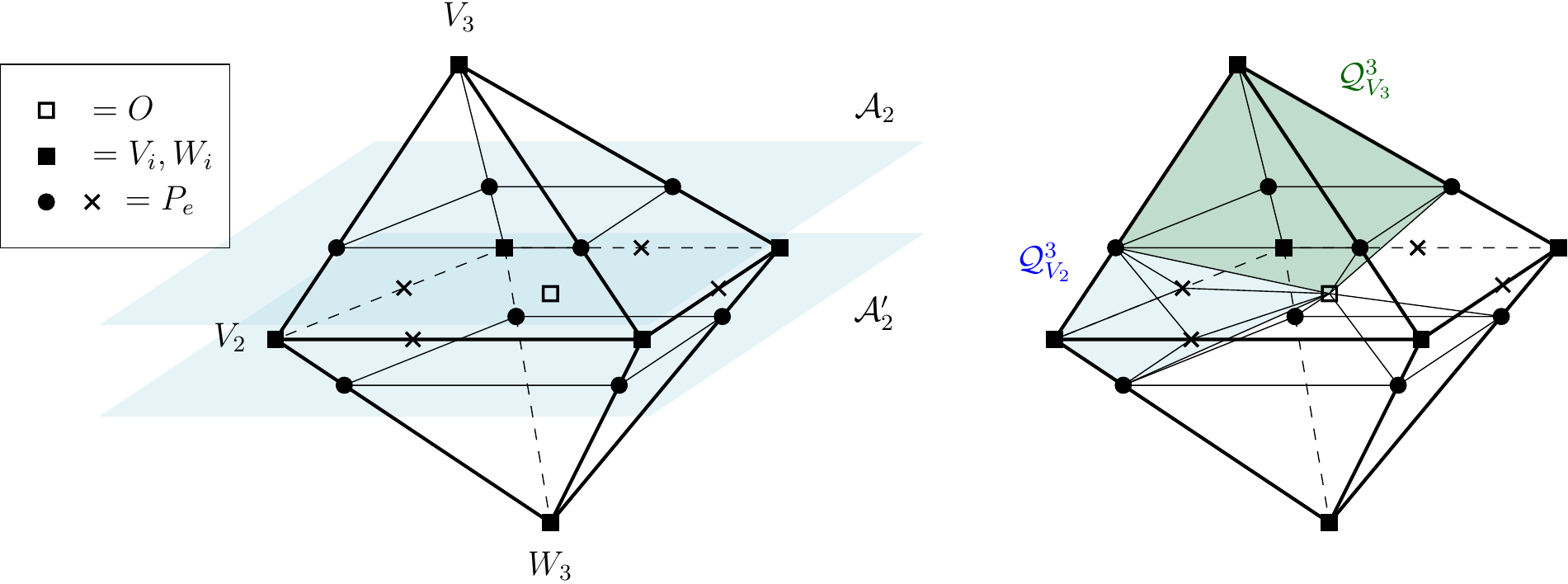}
\caption{Some of the cross-polytopes constructed in \Cref{lem:cubeoctahedron} for $d=3$.}
\label{fig: 4.6}
\end{figure}
\Cref{fig: 4.6} offers a partial visualization of \Cref{lem:cubeoctahedron}. In the following lemma we construct the cross-polytopes of type 3. 
\begin{lemma}\label{lem:type3}
	For any configuration of points as in \Cref{lem:cubeoctahedron}, $\conv(X_i,\epsilon V_i, \{P_e: V_i\in e\})$ and $\conv(Y_i,\epsilon W_i, \{P_e: W_i\in e\})$ are $d$-dimensional cross-polytope for any $i=1,\dots,d$. 
\end{lemma}
\begin{proof}
    Since $X_i, V_i, \epsilon V_i$ and the origin are aligned, as long as we choose $\epsilon$ small enough, the intersection of $\{P_e: V_i\in e\})$ with the segment $V_iO$ is the same as its intersection with the segment with endpoints $X_i$ and $\epsilon V_i$. Since $\conv(V_i,O,\{P_e: V_i\in e\})$ is a $d$-dimensional cross-polytope, by \cref{lem:antipodal} $\conv(X_i,\epsilon V_i, \{P_e: V_i\in e\})$ is a cross-polytope.
\end{proof}


Next we construct the cross-polytopes of type 1 and 2. We call a \emph{truncation} of a simplex $\mathcal{S}$ w.r.t. a vertex $v$ a polytope obtained intersecting $\mathcal{S}$ with the halfspace defined by an hyperplane separating $v$ from the other vertices, which does not contain $v$.
\begin{lemma}\label{lem:type12}
	Let $\po{P}$ be a truncation of a $d$-simplex $\po{S}$. Denote by $\po{F}$ the facet opposite the truncated vertex, and $\po{F}'$ the new facet introduced by the truncation. For any hyperplane $h$ which separates $\po{F}$ and $\po{F}'$ and for any choice of $d$ points $P_{\po{G}}$, one in the interior of each facet $\po{G}$ of the $(d-1)$-simplex $h\cap\po{P}$, $\conv(\{P_{\po{G}}\},\po{F})$ and $\conv(\{P_{\po{G}}\},\po{F}')$ are $d$-dimensional cross-polytopes.
\end{lemma}
\begin{proof}
	For $\po{G}=\conv(V(\po{F})\setminus v)$ we denote with $P_{\po{G}}$ the chosen point on the $(d-2)$-face of $h\cap\po{P}$ which corresponds to $\po{G}$. Consider $\conv(\{P_{\po{G}}\},\po{F})$. Clearly $\{P_\po{G}\}$ and $\po{F}$ are facets, since $\po{F}$ lies in one of the halfspace defined by the supporting hyperplane $h$ of $\{P_\po{G}\}$. 
	
	For any vertex $V$ of $\po{F}$, the segment $\conv(V,P_{\po{G}})$ with $\po{G}=\conv(V(\po{F})\setminus V)$ intersects the interior of $\po{S}$, because it joins $V$ with a point in the interior of $\conv(V(\po{S})\setminus V)$. Moreover $\conv(V,P_{\po{G}})$ is also contained in one of the halfspaces defined by $h$. Therefore $\conv(V,P_{\po{G}})$ is not a face of $\conv(\{P_{\po{G}}\},\po{F})$ for any $(d-2)$-face $\po{G}=\conv(V(\po{F})\setminus V)$, which implies that $\conv(\{P_{\po{G}}\},\po{F})$ is a $d$-cross-polytope by \Cref{lem:antipodal}. The proof for $\conv(\{P_{\po{G}}\},\po{F}')$ is analogous. 
\end{proof}	
\Cref{fig:prism} shows the cross-polytopes $\conv(\{P_{\po{G}}\},\po{F})$ (blue) and $\conv(\{P_{\po{G}}\},\po{F}')$ (red) in the $3$-dimensional case.
\begin{figure}[h]
\includegraphics[scale=0.9]{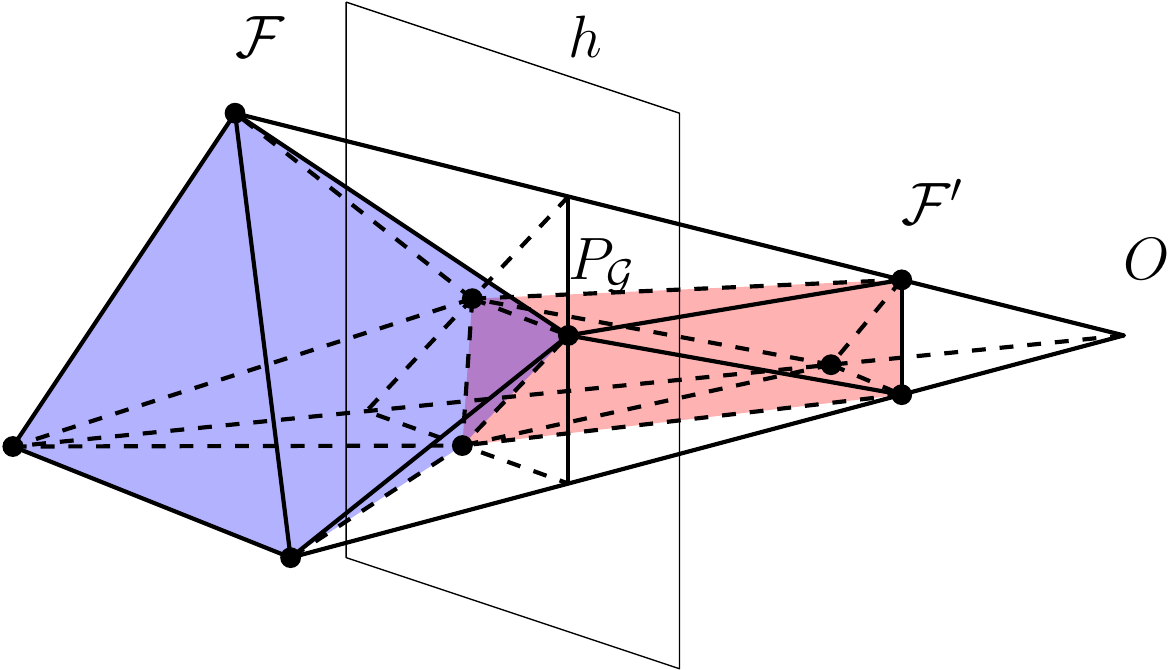}
\caption{An illustration of \Cref{lem:type12} in the $3$-dimensional case.}\label{fig:prism}
\end{figure}
We can finally put together the pieces to prove \Cref{prop:bipyramid}:
\begin{proof}[Proof of \Cref{prop:bipyramid}]
    We assume that the barycenter of the equatorial simplex of $\mathcal{B}$ is the origin $O$. We proceed as the outline at the beginning of this section. First, \Cref{lem:from_simplex_to_cp} allows us to construct an octahedron $\po{C}$ whose vertices lie on the segments $OX_i$ and $OY_i$, $i=1,\dots,3$. \Cref{lem:type3} and \Cref{lem:type12} guarantee the existence of a choice of points $P_e$ on the edges of $\po{C}$ and of a number $\epsilon\in(0,1)$ such that the polytopes
    \begin{itemize}
        \item $\conv(\mathcal{F}, \{P_e: e\subseteq \mathcal{F} \})$ (type 1, 8 polytopes),
        \item $\conv(\mathcal{F'}, \{P_e: e
        \subseteq \mathcal{F'} \})$ (type 2, 8 polytopes),
        \item $\conv(X_i,\epsilon V_i, \{P_e: V_i\in e\})$ and $\conv(Y_i,\epsilon W_i, \{P_e: W_i\in e\})$ (type 3, 6 polytopes), 
    \end{itemize}
    are octahedra, for any facets $\po{F}$ and $\po{F'}$ of $\mathcal{B}$ and $\epsilon\po{C}$ respectively. The statement follows letting $\Delta$ to be the cross-polytopal complex generated by these 22 octahedra, together with $\epsilon \po{C}$. Indeed all of the octahedra lie inside of $\po{B}$ and it is immediate to check that the intersection of two octahedra in $\Delta$ is a face of both. Moreover, since every $2$-dimensional face of $\Delta$ which is not on the boundary of $\po{B}$ lies in exactly two octahedra, we have that $\left|\Delta\right|=\mathcal{B}$.  
\end{proof}
We are now ready to prove \Cref{thm:3-colorable}. Recall that for a simplicial $3$-polytope Euler relation and a double counting argument show that the number of edges and $2$-faces are uniquely determined by the number of vertices. In particular we have that $f_2(\mathcal{P})=2(f_0(\po{P})-2)$.
\begin{proof}[Proof of \Cref{thm:3-colorable}]
    For any triangulation of $\po{P}$ with the conditions of \Cref{lem:triangulation}, the $3$-simplices in the triangulation can be pairwise matched in $f_2(\mathcal{P})/2$ many generalized bipyramids. It is important to observe that the second condition in \Cref{lem:triangulation} ensures that we can consider the barycenters of one of the edge for each equatioral simplex, so that for each bipyramid exactly one edge on the equator is subdivided. By \Cref{prop:bipyramid} there exists a geometric cross-polytopal subdivision on $23$ octahedra for each of the generalized bipyramids. The union of these $f_2(\mathcal{P})/2$ many geometric cross-polytopal complexes gives a proper cross-polytopal subdivision $\Delta$ of $\po{P}$ with $f_3(\Delta)=23f_2(\mathcal{P})/2=23(f_0(\mathcal{P})-2)$ octahedra. 
\end{proof}
\begin{remark}\label{rem:Eran}
    In \Cref{sect:bipyramid}, we reduced the problem of finding an octahedralization of a balanced $d$-polytope to that of the generalized bipyramid.
    However, we can apply verbatim the same construction described in \Cref{sect:dim3} directly to a simplex. Indeed, if $\mathcal{S}$ is a simplex, and $F_0\subset\dots\subset F_{d-1}$ a flag of faces in its boundary, then $\Gamma'=\text{sd}_{F_0}\circ\text{sd}_{F_1}\circ\dots\circ\text{sd}_{F_{d-2}}(\partial\po{S})$ is a subdivision of $\partial\mathcal{S}$ that is combinatorially isomorphic to the boundary of a $d$-dimensional cross-polytope. All the results in \Cref{sect:dim3} carry over. Decomposing a balanced $3$-polytope in this way would produce a cross-polytopal decomposition with $23f_2(\mathcal{P})=46(f_0(\mathcal{P})-2)$ many octahedra, that is twice as many as the strategy using bipyramids yields.
\end{remark}
\Cref{thm:3simplex} follows directly from this remark.
\section{Concluding questions}
    This paper leaves \Cref{quest:main} open in the case $d\geq 4$. The reason is that we are not aware of the existence of a polytope with 'many' cross-polytopal facets in dimensions higher than $4$, in analogy with the $24$-cell in dimension $4$, whose Schlegel diagram would be a starting point of our construction.
    
    Indeed, in dimensions higher than $3$, before embarking on the geometrical question, one might want to understand whether the following, combinatorial statement holds. Recall that a pure CW-complex is \emph{strongly regular} if the intersection of two cells is a single (possibly empty) cell. A $d$-dimensional strongly regular cross-polytopal complex is therefore a pure, strongly regular $d$-dimensional CW-complex in which all maximal cells are combinatorially isomorphic to $\Diamond_d$. 
    
\begin{question}\label{quest:main_strongly_reg}
    Is any boundary of a balanced $d$-polytope realizable as the boundary of a $d$-dimensional strongly regular cross-polytopal complex homeomorphic to a $d$-ball, for $d\geq 4$? 
\end{question}    
   
    Due to convexity, polytopal complexes are strongly regular CW-complexes, and so \Cref{thm:3-colorable} provides a positive answer in the three dimensional case. A negative answer to this question would of course imply a negative answer to \Cref{quest:main}.
    
    For $d\geq 4$ there are combinatorial $(d-1)$-spheres (i.e., simplicial complexes PL-homeomorphic to the boundary of a $d$-simplex) which cannot be realized as the boundary of a polytope. It is not always easy to check whether a sphere has this property or not. Therefore we can generalize \Cref{quest:main_strongly_reg} to the following one, which is interesting in its own right and may be easier to answer.
\begin{question}
    Is any balanced combinatorial $(d-1)$-sphere realizable as the boundary of a $d$-dimensional strongly regular cross-polytopal complex homeomorphic to a $d$-ball, for $d\geq 4$? 
\end{question}    
    A negative answer to this question would not however necessarily give an obstruction to \Cref{quest:main}. As mentioned in the introduction, \cite[Theorem 3.1]{IKN} answers this question positively for balanced combinatorial (even simplicial) spheres without the assumption of strong regularity. In a sense, strongly regular cross-polytopal complexes are an intermediate step between the complexes considered in \cite{IKN} and geometric cross-polytopal complexes.

\section*{Acknowledgement}
We would like to thank Martina Juhnke-Kubitzke for suggesting the problem to us and for insightful discussions and comments on the manuscript. Many thanks to Francisco Santos for interesting discussions and comments on the manuscript, and in particular pointing out the decomposition of the octahedron as the Schlegel diagram of the $24$-cell. We thank Eran Nevo for carefully listening to our presentation and suggesting the simplification of \Cref{rem:Eran}, which allows us to give \Cref{thm:3simplex} in its final form. A thank you also to Hannah Sj{\"o}berg for pointing out \Cref{lem:antipodal}. 

\bibliographystyle{alpha}
\bibliography{bibliography}
\end{document}